\documentclass[11pt, twosideside]{amsart}
\usepackage{amsmath}
\usepackage{amsthm}
\usepackage{amssymb}
\usepackage{amsmath}
\usepackage{amsthm}
\usepackage{amssymb}

\usepackage{mathrsfs}
\usepackage{amsbsy}
\usepackage{fancyhdr}
\usepackage{enumerate}
\usepackage{amscd}
\usepackage[all]{xy}
\usepackage{graphicx}
\usepackage{color}
\usepackage{hyperref}
\hypersetup{
    colorlinks,
    citecolor=black,
    filecolor=black,
    linkcolor=black,
    urlcolor=black
}
\usepackage[all]{xy}
\usepackage{graphicx}
\usepackage{pgfplots}

\theoremstyle{plain}
\newtheorem{lemma}{Lemma}

\newtheorem{proposition}[lemma]{Proposition}
\newtheorem{thm}[lemma]{Theorem}
\newtheorem{theorem}[lemma]{Theorem}
\newtheorem{definition}[lemma]{Definition}
\newtheorem{corollary}[lemma]{Corollary}
\numberwithin{equation}{section}

\newtheorem{theoremA}{Theorem}[section]
\newtheorem{propositionA}[theoremA]{Proposition}
\newtheorem{lemmaA}[theoremA]{Lemma}

\theoremstyle{remark}

\theoremstyle{definition}
\newtheorem{remark}[lemma]{Remark}
\newtheorem{example}[lemma]{Example}

\def\mB{\mathfrak{m}_{_B}}

\newtheorem*{ack}{Acknowledgements}

\usepackage[tmargin=1in,bmargin=1in,lmargin=1in,rmargin=1in]{geometry}
\begin{document}
\title{Static potentials and area minimizing hypersurfaces}
\author{Lan-Hsuan Huang}
\address{Department of Mathematics, University of Connecticut, Storrs, CT 06269, USA }
\email{lan-hsuan.huang@uconn.edu}
\author{Daniel Martin}
\address{Department of Mathematics, University of Connecticut, Storrs, CT 06269, USA }
\email{daniel.martin@uconn.edu}
\author{Pengzi Miao}
\address{Department of Mathematics, University of Miami, Coral Gables, FL 33146, USA}
\email{pengzim@math.miami.edu}
\thanks{The first two authors were partially supported by the NSF through grant DMS~1452477. The third author was partially supported by Simons Foundation Collaboration Grant for Mathematicians \#281105.}
\begin{abstract} 
We show that if an asymptotically flat manifold with horizon boundary admits a global static potential, then the static potential must be zero on the boundary. We also show that if an asymptotically flat manifold with horizon boundary admits an unbounded static potential in the exterior region, then the manifold must contain a complete non-compact area minimizing hypersurface. Some results related to the Riemannian positive mass theorem and Bartnik's quasi-local mass are obtained. 
\end{abstract}
\maketitle

The purpose of this paper is to study the interplay between  static potentials  and minimal hypersurfaces of an asymptotically flat  manifold. 

We state the main results. See Appendix~\ref{section:asymptotically-flat} and Appendix~\ref{section:static} for precise statements of terms used below.

\begin{theorem}\label{theorem:bounded}
Let $n\ge 3$. Let $(M, g)$ be an $n$-dimensional asymptotically flat manifold with horizon boundary. Suppose $(M, g)$ admits a static potential $V$. Then $V$ is zero on $\partial M$.  

As a direct consequence, if $V$ is bounded, then V is either positive or negative everywhere in the interior of M.
\end{theorem}

The motivation for the above theorem comes from the rigidity of the Riemannian positive mass theorem. In fact, combining with the work  of J.~Corvino on scalar curvature deformation~\cite{Corvino:2000}, the work of G.~Galloway and P.~Miao on static potentials~\cite[Theorem 4.1]{Galloway-Miao:2017}, and the rigidity result of O.~Chodosh and M.~Eichmair  \cite[Theorem 1.6]{Carlotto-Chodosh-Eichmair:2016},  the theorem  gives another proof to the rigidity of Riemannian positive mass theorem for asymptotically flat manifolds with horizon boundary in  three dimensions. We include the proof in Section~\ref{section:rigidity}.

We also obtain the following generalization of the result of Galloway and Miao~\cite[Theorem 4.1]{Galloway-Miao:2017}. Here, we only assume that the static potential is defined in an exterior region. 
\begin{theorem}\label{theorem:unbounded}
Let $3\le n\le 7$. Let $(M, g)$ be an $n$-dimensional asymptotically flat manifold. Suppose the boundary of $M$ is either empty or a disjoint union of smooth minimal hypersurfaces.  If one of the asymptotically flat ends admits an unbounded static potential, then there is a complete, non-compact, area minimizing hypersurface in $M$. \end{theorem}

In the above theorem, we do not assume the scalar curvature of $g$ to be everywhere nonnegative. In the proof, the complete area minimizing hypersurface is obtained as a limit of a sequence of Plateau solutions, and it is a well-known fact that the limiting hypersurface is smooth in dimensions $3\le n\le 7$.

If $n=3$ and the scalar curvature of $g$ is nonnegative in $M$, by the result of Chodosh and Eichmair  \cite[Theorem 1.6]{Carlotto-Chodosh-Eichmair:2016}, an immediate consequence of Theorem~\ref{theorem:unbounded} gives the following statement.  

\begin{corollary}
Let $(M, g)$ be a three-dimensional asymptotically flat manifold with horizon boundary.  Suppose $(M, g)$ has nonnegative scalar curvature. If the exterior region of $(M, g)$ admits an unbounded static potential, then $(M, g)$ is isometric to Euclidean space. 
\end{corollary}

We  include other results related to Bartnik's quasi-local mass in Section~\ref{section:2} and Section~\ref{section:3}.

\begin{ack}
The first named author is  grateful to Jim Isenberg for support. We thank  Lucas Ambrozio, Justin Corvino, and Jeff Jauregui for comments on an earlier version of the paper. We also thank the referee for helpful suggestions to improve the presentation. 
\end{ack}

\section{Proof of  Theorem~\ref{theorem:bounded}}

To establish the relation between locally area minimizing hypersurfaces and a static potential, we need the following lemma. Recall that in the Appendix \ref{section:static}, we define the static potential $V$ as a non-trivial solution  to the following static equation
\[
	 -(\Delta V)g + \nabla^2 V - V Ric=0.
\]
\begin{lemma}[{\cite[Equations (9)-(14)]{Miao:2005}}]\label{lemma:stability}
Let $(\Omega, g)$ be an $n$-dimensional Riemannian manifold. Suppose that $\Omega$ admits a static potential $V$. Let $\Sigma$ be a  closed, connected, stable minimal hypersurface in $\Omega$. Then we have the following: 
\begin{enumerate}
\item Either $V>0$ or $V<0$ on $\Sigma$, unless $V$ is identically zero on $\Sigma$. 
\item $\Sigma$ is totally geodesic.
\end{enumerate}
\end{lemma}
\begin{proof}
By the stability inequality, for any $\phi\in C^1(\Sigma)$, 
\[
	\int_\Sigma |\nabla_\Sigma \phi|^2 \, d\sigma\ge \int_\Sigma\left( |A|^2 + Ric(\nu, \nu)\right) \phi^2\, d\sigma \ge \int_\Sigma Ric(\nu, \nu) \phi^2\, d\sigma, 
\]
where $\nu$ is a unit normal vector field to $\Sigma$ and $d\sigma$ is the $(n-1)$-volume measure of hypersurfaces. It implies that the first eigenvalue of the operator $\Delta_\Sigma + Ric(\nu, \nu)$ is non-positive, where $\Delta_\Sigma$ is the induced Laplacian. On the other hand, since $\Sigma$ is minimal, the restriction of the static potential $V$ on $\Sigma$ satisfies $\Delta V= \Delta_\Sigma V +\nabla^2 V(\nu, \nu)$. By the static equation of $V$
\begin{align}\label{equation:Ric}
	0= \Delta_\Sigma V + \nabla^2 V(\nu, \nu) - \Delta V = \Delta_\Sigma V +Ric(\nu, \nu)  V.
\end{align}
It implies either $V$ is identically zero or $V$ is  the first eigenfunction with the zero eigenvalue. If $V$ is zero on $\Sigma$, then $\Sigma$ lies in the  zero set of $V$ which is totally geodesic. If $V$ is the first eigenfunction, then $V$ does not vanish on $\Sigma$.  Substituting $V$ in the stability inequality, we obtain $\int_\Sigma |A|^2 V^2 \le 0$. Thus $|A|\equiv 0$ and $\Sigma$ is also totally geodesic. 

\end{proof}

If, furthermore, $\Sigma$ is locally area minimizing,  a splitting result is obtained by adapting the argument of Galloway in three dimensions \cite[Lemma 3]{Galloway:1993}. We note that the argument of Galloway is also extended in \cite[Proposition 14 and Appendix B]{Ambrozio:2015}, which covers some of the following results in three dimensions. 

\begin{proposition}\label{proposition:splitting}
 Let $(\Omega, g)$ be an $n$-dimensional Riemannian manifold with the scalar curvature $R_g = 0$. Suppose that $\Omega$ admits a static potential $V$. Let $\Sigma$ be a  locally area minimizing,  closed, connected hypersurface in $\Omega$. Suppose $V$ is not identically zero on $\Sigma$. Then there is a subset $U$ of $\Omega$ and a diffeomorphism $\Phi: \Sigma\times [0, \epsilon) \to U$  so that the following  holds: 
\begin{enumerate}
\item  The $(n-1)$-volume of hypersurfaces $\Sigma_t:=\Phi(\Sigma\times \{ t\})$ is constant in $t$. 
\item The induced scalar curvature $R_\Sigma$ of $\Sigma_t$ is zero and $V$ is constant on $\Sigma_t$ for each $t$. 
\item The Ricci curvature of $g$ is zero on $U$.
\end{enumerate}

\end{proposition}
 \begin{proof}
 By Lemma~\ref{lemma:stability}, we  may without loss of generality assume $V>0$ on $\Sigma$. Consider the deformation $\Phi:\Sigma\times [0, \epsilon) \to \Omega$ given by the normal exponential map with respect to the conformally modified metric $V^{-2} g$ in  a collar neighborhood of $\Sigma$ where $V>0$.  Let $\Sigma_t = \Phi(\Sigma\times \{ t\})$ and note $\Sigma_0 = \Sigma$.  Let $H(\cdot, t), A(\cdot, t)$ be the mean curvature and second fundamental form of $\Sigma_t$ in  the metric $g$, respectively. Lemma~\ref{lemma:monotonicity} implies that  $H(\cdot, t)\ge 0$ for $t\in (0, \epsilon)$. From the first variation of area, we have
 \[
 	|\Sigma_t| - |\Sigma_0| = \int_0^t\left( -\int_{\Sigma_s} VH(\cdot, s) \, d\sigma \right) ds.
 \]
For $\epsilon$ sufficiently small, $\Sigma$ is locally area minimizing. Therefore, the above identity implies that the mean curvature of $\Sigma_t$ cannot be strictly positive for $t<\epsilon$. Hence $H(\cdot, t)\equiv 0$ and the $(n-1)$-volume of $\Sigma_t$ is a constant. By Lemma~\ref{lemma:monotonicity} again, $A(\cdot, t) \equiv 0$ and  $\Sigma_t$ is totally geodesic for $t\in [0, \epsilon)$ with respect to the metric $g$.

Furthermore, using the first variation of the second fundamental form (see, for example, \cite[p. 993]{Carlotto-Chodosh-Eichmair:2016} and the references therein), we obtain, for vectors $X, Y$ tangential to $\Sigma_t$,  
 \[
 	\nabla_{\Sigma}^2 V(X, Y) + Rm(\nu, X, Y, \nu) V=0,
 \]
 where $\nabla_{\Sigma}$ denotes the connection  of $\Sigma_t$, $\nu$ is a unit normal vector to $\Sigma_t $ (both with respect to the metric $g$), and $Rm$ is the Riemann curvature tensor  of $(\Omega, g)$ (with the sign convention that the Ricci tensor is the trace on the first and fourth components of $Rm$).   Because $\Sigma_t$ is totally geodesic,  $\nabla_{\Sigma}^2 V(X, Y) = \nabla^2 V(X, Y)$ for tangential vectors $X, Y$. Then by the static equation \eqref{equation:static}, the assumption that $R_g=0$, and $V>0$,  we obtain $Ric(X, Y) = -Rm(\nu, X, Y, \nu)$. For an orthonormal frame $\{ E_i \}$ on $\Sigma_t$, 
 \begin{align*}
 	Ric(X, Y)& = Rm(\nu, X, Y, \nu) + \sum_i Rm(E_i, X, Y, E_i)\\
	&= - Ric(X, Y) +   Ric_\Sigma (X, Y),
 \end{align*}	
 where we also use the Gauss equation in the second equality and denote by $Ric_{\Sigma} $ the Ricci  tensor of $\Sigma_t$ induced from $g$. It gives that, for all tangential vector fields $X, Y$ to $\Sigma_t$, 
 \begin{align} \label{equation:Ricci}
	 Ric(X, Y) = \frac{1}{2} Ric_{\Sigma}(X, Y)
 \end{align}
 and hence, combining the previous formulas gives 
\begin{align}
	\nabla_{\Sigma}^2 V &=  \frac{1}{2} VRic_{\Sigma}  \label{equation:induced-static-1}\\
	\Delta_\Sigma V &= \frac{1}{2}VR_\Sigma,\label{equation:induced-static-2}
\end{align}
where $R_\Sigma$ denotes the scalar curvature of $\Sigma_t$. Take the divergence of \eqref{equation:induced-static-1} on $\Sigma_t$ and note $\mathrm{div}_\Sigma \left(\nabla_\Sigma^2 V \right)=d(\Delta_\Sigma V) + Ric_\Sigma\cdot \nabla_\Sigma V $, where the dot in the last term denotes tensor contraction. Hence, we derive that, on each $\Sigma_t$,
\begin{align*}
	 0 &= d(\Delta_\Sigma V) + Ric_{\Sigma}\cdot \nabla_\Sigma V -\frac{1}{2}  \left(\frac{1}{2} V dR_\Sigma  + Ric_{\Sigma}\cdot \nabla_\Sigma V\right)\\
	 &=\frac{1}{2}d( R_\Sigma V )+ \frac{1}{2}  Ric_{\Sigma}\cdot \nabla_\Sigma V - \frac{1}{4}V dR_\Sigma  \\
	 &= \frac{1}{4} VdR_\Sigma + \frac{1}{2} R_\Sigma dV  + V^{-1}\nabla_\Sigma^2 V \cdot \nabla_\Sigma V\\
	 	&= \frac{1}{4}V^{-1} d( R_\Sigma V^2 + 2|\nabla_\Sigma V|^2 ).
\end{align*} 
This implies that $R_\Sigma V^2 + 2|\nabla_\Sigma V|^2 $ is constant on each $\Sigma_t$ and in fact,  by \eqref{equation:induced-static-2}, 
\[
	R_\Sigma V^2 + 2|\nabla_\Sigma V|^2 = 0.
\]
It gives $R_\Sigma\le 0$. On the other hand, by \eqref{equation:induced-static-2},
\[
	\int_{\Sigma_t} R_\Sigma\, d\sigma =2 \int_{\Sigma_t} V^{-2} |\nabla_\Sigma V|^2 \, d\sigma\ge 0.
\]
Hence $R_\Sigma =0$  and $V$ is constant on $\Sigma_t$ for each $t\in [0,\epsilon)$. By \eqref{equation:induced-static-1}, $Ric_\Sigma=0$ and by \eqref{equation:Ricci} $Ric(X, Y)=0$ for vectors tangential to $\Sigma_t$.  By the Codazzi equation, $Ric(X, \nu)$ is zero, and by the Gauss equation, $Ric(\nu, \nu)$ is zero. Thus, the Ricci tensor is zero in $U$.
\end{proof}

\begin{proof}[Proof of Theorem~\ref{theorem:bounded}]
Note that the scalar curvature of $g$ is constant on $M$ and hence must be zero, by asymptotic flatness. 
If $V$ is not zero on $\partial M$, by Proposition~\ref{proposition:splitting}, a collar neighborhood of $\partial M$ in $M$ splits as a foliation of minimal hypersurfaces. It contradicts that $M$ contains no closed minimal hypersurfaces other than $\partial M$. We also note that since $V$ is not identically zero, each component of the zero set of $V$ is a regular hypersurface, and hence $\partial M$ is itself a connected component of the zero set. 

For the rest of the proof, we assume $V$ is bounded. By Proposition~\ref{proposition:asymptotics}, $V$ has the following expansion on each end $N_k$,  for a nonzero constant $A_k$, 
\[
	V (x)=  A_k+ O(|x|^{2-n}).
\]
We may assume $A_1>0$ (otherwise, consider $-V$). It implies that $A_k>0$ for all other $k$; otherwise, the zero set of $V$ is nonempty in the interior $M$, which would  imply that $M$ has a closed minimal hypersurface other than $\partial M$. Therefore, by the strong maximum principle for harmonic functions, $V>0$ in $M$. 
\end{proof}
We remark that in the preceding proof, we can further  apply the Hopf boundary point lemma to conclude that, for $V>0$ in $M$, the normal derivative $\nabla_\nu V >0$ on $\partial M$ with respect to the normal vector $\nu$ to $\partial M$ pointing into $M$.

\section{Proof of Theorem~\ref{theorem:unbounded}}

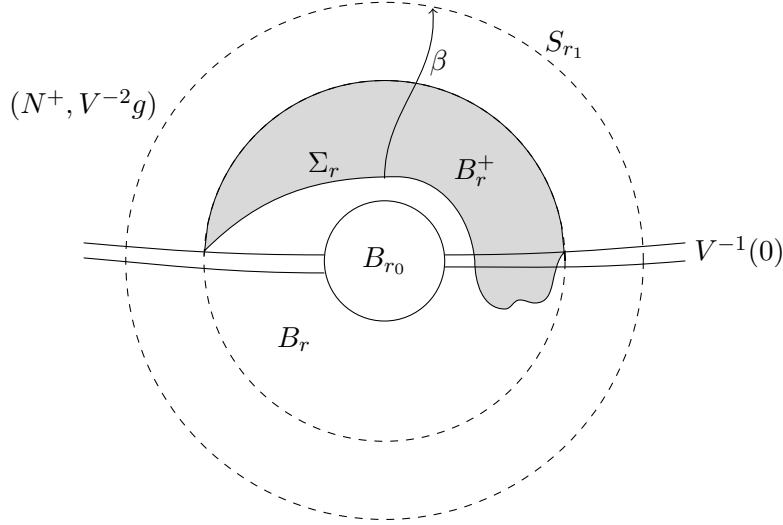
\begin{figure}
\begin{tikzpicture}[scale=0.8]
\draw[fill=gray!30!white] 
plot[smooth,samples=100]  (4,0.16) to [out=45,in=180] (7.2,1.4) --
(7.2,1.4) to [out=0, in=95] (8.5,0) -- 
(8.5,0) to [out =270,in=180 ] (9,-0.8) -- 
(9,-0.8) to [out=0,in = 180] (9.25,-0.65)--  
(9.25,-0.65) to [out=0, in= 180] (9.5,-0.7)--  
(9.5,-0.7) to [out=0,in = 225] (10,0.16) --
plot[dashed ,samples=100,domain=10:4] (\x, {sqrt(9 - (\x - 7)*(\x - 7))}) ;
\node at (6,1.6) {$\Sigma_r$};
\node at (2, 2.6) {$(N^+, V^{-2}g)$};
\node at (5.5,-1.3) {$B_r$};
\node at (7,0) {$B_{r_0}$};
\node at (12.9, 0.15) {$V^{-1}(0)$};
\node at (10,3.6) {$S_{r_1}$};
\node at (8.5,1.5) {$B_r^+$};
\node at (7.9,3.3) {$\beta$};
\draw (7,0) circle [radius = 1];
\draw [dashed] (7,0) circle [radius = 3];
\draw [dashed] (7,0) circle [radius = 4.3];
\draw [->](7, 1.37) to [out = 90, in =-85 ] (7.8, 4.2);
\draw (2,0.3) to [out = 355, in= 180] (6,0.1);
\draw (8,0.1) to [out = 0, in= 185] (12,0.3);
\draw (2,0.05) to [out = 355, in = 180] (6,-0.2);
\draw (8,-0.1) to [out = 0, in = 185] (12,0);
\end{tikzpicture}

\caption{An illustration for the contradictory argument in the proofs of Theorem~\ref{theorem:unbounded} and Lemma~\ref{lemma:plateau}. If the Plateau solution $\Sigma_r$ was disjoint from $B_{r_0}$, it would have separated $B_{r_0}$ from infinity in $N^+$. Then a minimizing geodesic $\beta$ in the interior with respect the modified metric $V^{-2}g$ from $\Sigma_r$ to a large coordinate sphere $S_{r_1}$ exists, which leads to a contradiction.}
\end{figure}

The following observation is due to G. Galloway and P. Miao in \cite{Galloway-Miao:2017}. We recall that $B_r$ denotes the large coordinate ball of radius $r$ with respect to the chart at infinity, as defined in Appendix~\ref{section:asymptotically-flat}.

  \begin{lemma}[{Essentially \cite[Theorem 3.1]{Galloway-Miao:2017}}]\label{lemma:plateau}
Let $n\ge 3$ and let $(M, g)$ be an $n$-dimensional asymptotically flat manifold. Let $N$ be one of the ends. Suppose  $N\setminus B_{r_0}$ admits a static potential $V$ for some $r_0>0$. Let $N^+$ be an unbounded component of the complement of $V^{-1}(0)$ in $N\setminus B_{r_0}$. Let  $\Sigma$ be a  compact two-sided minimal hypersurface in $N^+$ with boundary either on $V^{-1}(0)$ or empty. Then $\Sigma$ cannot separate $B_{r_0}$ from infinity in $N^+$; that is, every unbounded component of the complement of $\Sigma$ in  $N^+\cup B_{r_0}$ contains $B_{r_0}$.
 \end{lemma}
\begin{proof}
Suppose on the contrary that there is an unbounded component $\Omega$ of the complement of $\Sigma$ in $N^+\cup B_{r_0}$ that does not contain $B_{r_0}$.   Note that $V$ is globally defined and nonzero on $\Omega$, and $\partial \Omega$ consists of $\Sigma$ and a subset of $V^{-1}(0)$. We may assume $V>0$ in $\Omega$; otherwise consider $-V$. We shall consider geodesics in the modified metric $\bar{g}=V^{-2} g$ emitting from $\Sigma$ into $\Omega$. It is shown in \cite[Lemma 3.1]{Galloway-Miao:2017} that those geodesics cannot reach the zero set of $V$ in the finite $\bar{g}$ length, and any two disjoint points in the zero set have infinite $\bar{g}$ distance. The rest of the argument follows from \cite[Theorem 3.1]{Galloway-Miao:2017} which we briefly summarize below.  Consider a large coordinate sphere $S_r$ that intersects  $ \Omega$ and is disjoint from $\Sigma\cap \Omega$. There exists a \emph{minimizing} geodesic in the modified metric $\bar{g}$ emitting from the interior of $\Sigma$ in $\Omega $ that reaches $S_r\cap \Omega$.  In a tubular neighborhood of the geodesic, we consider the level set of the distance function  with respect to the $\bar{g}$ metric from $\Sigma$. By the monotonicity formula Lemma~\ref{lemma:monotonicity}, those hypersurfaces have nonpositive mean curvature in the metric $g$ (with respect to $-\nu$, where $\nu$ in as in Lemma~\ref{lemma:monotonicity}). It ultimately leads a contradiction to the convexity of large coordinate spheres and the maximum principle. 
\end{proof}

 \begin{proof}[Proof of Theorem~\ref{theorem:unbounded}]
Let $V$ be an unbounded static potential on one of the ends, say $N$. By Proposition~\ref{proposition:asymptotics}, $V$ is asymptotic to a linear combination of coordinate functions on the end $N$. By rotating the asymptotically flat coordinate chart of $N$ and rescaling $V$ if necessary, we assume  $V(x)$ is asymptotic to $x^n$.  By Lemma~\ref{lemma:zero}, there is $r_0$ sufficiently large so that each component of $V^{-1}(0)$ is a graph $x^n = f(x^1, \dots, x^{n-1})$ that intersects with $S_r$ in a nearly equatorial $(n-2)$ sphere for $r>r_0$. We may assume $r_0$ sufficiently large so that $S_{r_0}$ does not intersect any closed minimal hypersurfaces. 

 For $r> r_0$, consider the orientable Plateau solution $\Sigma_r$ whose boundary spans the intersection of $S_r$ and a component of $V^{-1}(0)$. We claim that $\Sigma_r$ must intersect $B_{r_0}$ for all $r>r_0$.  Suppose on the contrary that $\Sigma_r$ is disjoint from $B_{r_0}$. Since $\Sigma_r$ separates $B_r$, there is a component $B_r^+$ of the complement of $\Sigma_r$ in  $B_r$ that does not contain $B_{r_0}$. We may without loss of generality assume that $B_r^+$ contains the top portion of $S_r$ (otherwise, consider $-V$). Then we consider the \emph{top} component $N^+$ of the complement of $V^{-1}(0)$ in $N\setminus B_{r_0}$, i.e. the component containing all points with sufficiently large $x^n$-coordinate values. Notice that $\Sigma_r \cap N^+$ separates $B_{r_0}$ from infinity in $N^+$, as the complement of $\Sigma_r$ in  $N^+\cup B_{r_0}$ has only one unbounded component $\Omega$, and $\Omega$ cannot intersect the components of $B_r \setminus \Sigma_r$ other than $B_r^+$, by connectedness of $\Omega$.  It gives a contradiction to Lemma~\ref{lemma:plateau}. 
 
 Since $\Sigma_r$ intersects $S_{r_0}$  for all $r>r_0$ and $\{ \Sigma_r\}$ has a  uniformly local area bound, by standard geometric measure theory, a subsequence of $\Sigma_r$ converges to a nonempty  complete area minimizing hypersurface $\Sigma$ intersecting $S_{r_0}$ as $r\to \infty$. Since $S_{r_0}$ does not intersect any closed minimal hypersurface, $\Sigma$ is unbounded.  
 \end{proof}

\section{Applications}
\subsection{Rigidity of the Riemannian positive mass theorem}\label{section:rigidity}

The Riemannian positive mass theorem is due to R.~Schoen and S.~T.~Yau~\cite{Schoen-Yau:1979-pmt1, Schoen-Yau:1981-asymptotics, Schoen-Yau:2017}. Other proofs for the three-dimensional case can be found in \cite{Huisken-Ilmanen:2001} and \cite{Li:2016}. 

Here we use static potentials to give another proof of the rigidity of the Riemannian positive mass theorem in three dimensions. The argument for complete manifolds without boundary may have already been known to  the experts. Here we use Theorem~\ref{theorem:bounded} and extend the argument to asymptotically flat manifolds with horizon boundary.

\begin{theorem}\label{theorem:rigidity}
 Let $(M, g)$ be a three-dimensional asymptotically flat manifold with horizon boundary and nonnegative scalar curvature. Assume $g\in C^{4,\alpha}_{\mathrm{loc}}$. If one of the ends has zero ADM mass, then $(M, g)$ is isometric to Euclidean space. 
\end{theorem}
\begin{proof}

We first apply the argument similar to \cite[Theorem 8]{Corvino:2000} to show that every precompact open subset $\Omega$ in $M$ admits a static potential $V\in C^{4,\alpha}(\Omega)\cap C^{2,\alpha}(\overline{\Omega})$.  The only difference is that in order to keep the minimal boundary condition, we consider the conformal Laplacian with the  \emph{Neumann boundary condition} as follows.  Suppose on contrary there is a precompact open subset $\Omega$ which does not admit a static potential. By \cite[Theorem 1]{Corvino:2000} (see, also, Theorem~\ref{theorem:deformation}), there is a $C^{2,\alpha}_{\mathrm{loc}}$ metric $\bar{g}$ with positive scalar curvature in $\Omega$ such that $\bar{g}$  coincides with $g$ outside $\Omega$. Then \cite[Lemma 3.3]{Schoen-Yau:1979-pmt1} implies there exists a unique positive solution $u$ to $\Delta_{\bar{g}} u -\frac{1}{8} R_{\bar{g}} u=0$ in $M$ with $\frac{\partial u}{\partial \nu} = 0$ on $\partial M$ and $u(x)\to 1$ as $|x|\to \infty$ so that $u^4 \bar{g}$ has negative mass. It gives a contradiction to the positive mass inequality. 

Next we show that there is a global vacuum static potential $V\in C^{4,\alpha}_{\mathrm{loc}}(M)$. Let $B_k$  be an exhaustion sequence of coordinate balls of $M$. As shown in the previous paragraph, each $B_k$ admits a static potential $V_k\in C^{4,\alpha}(B_k)\cap C^{2,\alpha}(\overline{B_k})$.   For a fixed $r_0 >1$, we may normalize $V_k$ so that $\max_{S_{r_0} \cup \partial M} (|V_k| + |\nabla V_k|) =1$ for all $k> r_0$. Then by maximum principle (for $V_k$ in $B_{r_0}$) and by Proposition~\ref{proposition:asymptotics} (for $V_k$ in the annulus $B_{r}\setminus B_{r_0}$), we have $\sup_{B_r}|V_k| \le C r$ uniformly for all $k >r$. Thus, $\| V_k\|_{C^{2,\alpha}(B_r)} $ is bounded uniformly in $k$ by the Schauder estimate. By Arzela-Ascoli theorem and taking the diagonal sequence, we obtain a $C_{\mathrm{loc}}^{2,\frac{\alpha}{2}}$ limiting function $V$ in $M$ with $\sup_{S_{r_0} \cup \partial M} (|V| + |\nabla V|) =1$. Hence, $V$ is a nontrivial solution to the static equation, and   $V\in C^{4,\alpha}_{\mathrm{loc}}(M)$ by elliptic regularity. 

If $V$ is unbounded in $M$, then $M$ has a complete, non-compact, area minimizing surface by the result of Galloway and Miao~\cite[Theorem 4.1]{Galloway-Miao:2017} (or Theorem~\ref{theorem:unbounded}). The rigidity follows from the work of Chodosh and Eichmair \cite[Theorem 1.6]{Carlotto-Chodosh-Eichmair:2016}. 

We now discuss the case that $V$ is bounded. Note that since any two ends must be separated by a minimal surface and $M$ does not contain any minimal surfaces in its interior, $M$ has only one end with zero ADM mass.  If $V$ is bounded, by Proposition~\ref{proposition:asymptotics}, $V$ goes to a  constant $A$ at infinity. More specifically, 
\[
	V = A + o(|x|^{2-n}). 
\]
 Integrating $V\Delta V=0$  yields
\begin{align*}
	0 = \int_M V\Delta V \, d\mu&=-\int_M |\nabla V|^2 \, d\mu+ \lim_{r\to \infty}\int_{S_r} V\frac{\partial V} {\partial r} \, d\sigma+\int_{\partial M} V\frac{\partial V} {\partial r} \, d\sigma \\
	&=-\int_M |\nabla V|^2 \, d\mu,
\end{align*}
where we use that if $\partial M$ is nonempty, then $V$ is zero on $\partial M$ by Theorem~\ref{theorem:bounded}. We obtain $|\nabla V|=0$ in $M$. Hence $V$ is constant on $M$ and then $g$ is Ricci flat by the static equation~\eqref{equation:static}. For three-dimensional manifolds, it implies $g$ has zero sectional curvature and hence $(M,g)$ must be isometric to Euclidean space.

\end{proof}

We remark that Theorem~\ref{theorem:rigidity} and Theorem~\ref{theorem:bounded} are closely related to the uniqueness of static black holes, which says that an asymptotically flat manifold admitting a global static potential $V\ge 0$ must be isometric to a Schwarzschild metric.  However, we emphasize that our proofs to  Theorem~\ref{theorem:rigidity} and Theorem~\ref{theorem:bounded} are  \emph{independent} of the uniqueness of static black holes.  The proof of Bunting and Masood-ul-Alam~\cite{Bunting-Masood-ul-Alam:1987} and the later extensions in \cite{Chrusciel:2005, Miao:2005} \emph{use} Theorem~\ref{theorem:rigidity}. Although some results are obtained independent of Theorem~\ref{theorem:rigidity} in \cite{Israel:1968, Muller-zum-Hagen:1973, Robinson:1977}, more stringent conditions, such as positivity of $V$ in the interior  of $M$, $V=0$ on $\partial M$, and connectedness of $\partial M$, are assumed.

\subsection{The mass minimizer of Bartnik's quasi-local mass}\label{section:2}

We recall the definition of Bartnik's quasi-local mass 
proposed by R.~Bartnik~\cite{Bartnik:1989} 
and revised by H.~Bray \cite{Bray:2011} (see also \cite{Bray-Chrusciel:2004}) as follows.
\begin{definition} [\cite{Bray:2011, Bray-Chrusciel:2004}]
Let $(N, h)$ be a complete, asymptotically flat three-manifold with nonnegative scalar curvature.  Let $\Omega\subset N$ be a bounded subset such that $\partial \Omega$ is outer-minimizing in $(N,h)$. Let $\mathcal{PM}$ be the set of complete, asymptotically flat $3$-manifolds $(M, g)$ with nonnegative scalar curvature so that $(\Omega, h)$ isometrically embeds in $(M, g)$ and $\partial \Omega$ is outer-minimizing in $(M, g)$. The Bartnik quasi-local  mass is defined as 
$$ \mB (\Omega, h) =  \inf \left\{   m({M}, g ): ({M}, {g} ) \in \mathcal{PM}\right\},$$ 
where we   recall that $ m(M,  g)$ is the ADM mass of $(M,  g)$.
\end{definition}

In this definition, the outer-minimizing assumption of $\partial \Omega$ is imposed so that the Hawking mass of $\partial \Omega$ 
gives a lower bound of $\mB(\Omega, h)$. This follows from the inverse mean curvature flow argument of Huisken and Ilmanen  in the proof of  the Riemannian positive mass theorem~\cite{Huisken-Ilmanen:2001}.

\begin{example}

Let $(N, h) \in \mathcal{PM}$ be rotationally symmetric such that the scalar curvature of $h$ is identically zero outside  a bounded, rotationally symmetric subset $ \Omega$. 
(Such an $(N, h)$ can be easily constructed by an ODE method.)
By  uniqueness of rotationally symmetric solutions of the vacuum Einstein equations (or simply solving the ODE of the zero scalar curvature equation), $(N \setminus \Omega, h)$ is isometric to an exterior region of a Schwarzschild manifold.
In particular, $ \partial \Omega$ is strictly outer-minimizing in $(N, h)$ and 
 $ m (N, h)$ equals the Hawking mass of $\partial \Omega$. As a result, $ m(N, h)  = \mB(\Omega, h) $. 
Using the well-known facts about Schwarzschild manifolds, we see that this mass minimizer $(N, h)$ admits a static potential $V$ in the exterior region $N\setminus \Omega$, and $V$ approaches  a constant at infinity. 
\qed
\end{example}

Next we show that the above assertion on static potentials holds in general. Namely, if a suitable mass minimizer exists in $\mathcal{PM}$, then the exterior region of the mass minimizer admits a static potential that goes to a constant at infinity.

\begin{thm} \label{thm B}
Let $ \Omega \subset (N, h)$ be a bounded subset where $(N, h) \in  \mathcal{PM}$. Suppose there exists $(M, g) \in \mathcal{PM}$  such that  $ \partial \Omega$ is strictly outer-minimizing in $(M, g)$ and  $ \mB(\Omega, h) =m (M,g )$, 
then $(M\setminus \overline{\Omega}, g)$ admits a static  potential that goes to a constant at infinity. 
\end{thm}
\begin{remark}
An analogous result is also obtained by  M.~Anderson and J.~Jauregui \cite[Theorem 1.1]{Anderson-Jauregui:2016} using a different approach. Though, note that their definition of Bartnik's quasi-local mass is slightly different from ours because the minimization in their definition is taken over a large class of asymptotically flat $3$-manifolds. The first named author is very grateful to Jeff Jauregui for kindly explaining their proof. 
\end{remark}

\begin{proof} 
Using the mass minimizing property of $(M,g)$, a recent result of J.~Corvino~ \cite[Corollary 1.2]{Corvino:2017} shows that $(M\setminus \overline{\Omega}, g)$ admits a static potential $V$. (The strictly outer-minimizing assumption on $\partial \Omega$  
guarantees that the competitors produced in \cite[Corollary 1.2]{Corvino:2017} still lie in $\mathcal{PM}$.)
 By asymptotics of static potentials, either $V$ goes to a constant or $V$ is unbounded (see \cite{Beig-Chrusciel:1996} and \cite{Miao-Tam:2015}, or Proposition~\ref{proposition:asymptotics} below). If $V$ goes to a  constant, the claim follows. 

Now we assume that  $V$ is unbounded. Theorem~\ref{theorem:unbounded}  implies  that there is a complete, non-compact, area minimizing surface in $(M, g)$. We then invoke the rigidity result of Chodosh and Eichmair in \cite[Theorem 1.6]{Carlotto-Chodosh-Eichmair:2016} to conclude that $(M, g)$ is isometric to Euclidean space. Then it is obvious that the constant function is a static potential on Euclidean space. 
\end{proof}

\subsection{Geometric properties of a static extension}\label{section:3}

 Given a Riemannian metric $\gamma$ and a  function~$H$ on a $2$-sphere, we say that an asymptotically flat $3$-manifold
$(M, g)$ with boundary $\Sigma=\partial M$ is a \emph{static extension subject to the boundary data $(\gamma, H)$} if 
\begin{enumerate}
\item $\Sigma$ is diffeomorphic to the $2$-sphere, and the induced metric from $g$ on $\Sigma$ is isometric to $\gamma$.
\item The mean curvature of $ \Sigma$ with respect to the unit normal vector on $\Sigma$ pointing into $(M, g)$ is given by $H$.
\item $(M, g)$ admits a bounded static potential.
\end{enumerate}

Below we give the sufficient conditions  on $(\gamma, H)$ so that the static extension has no closed minimal surfaces that locally minimize the area.

\begin{thm}
Suppose the pair $(\gamma, H)$ satisfies 
$$
 H > 0  \ \ \mathrm{and} \ \  K_\gamma \ge  \frac14  H^2, 
$$
 where $K_\gamma$  denotes  the Gauss  curvature of $ \gamma$.
Then any static extension $(M, g)$ subject to the boundary data $(\gamma, H)$ does not have closed, locally area minimizing surfaces.
\end{thm}

\begin{proof}
Let $V $ be a bounded static potential on $(M, g)$. By Proposition~\ref{proposition:asymptotics} and normalizing, we may assume $V \to 1 $ at infinity. 

We now use  the argument in \cite[Proposition 3]{Miao:2005} to show that $ V > 0 $ in $M$.
By the static equation and recall $R_g=0$, we have the following identity on $\Sigma$:
\begin{equation}
\begin{split}
0 =  \Delta V 
= & \  \Delta_\Sigma V + H \frac{\partial V}{\partial \nu} + \nabla^2 V (\nu, \nu) \\
= & \  \Delta_\Sigma V + H \frac{\partial V}{\partial \nu} + Ric (\nu, \nu) V ,
\end{split}
\end{equation}
where $\nu$ is the unit normal vector on $ \Sigma$ pointing into $M$. By the Gauss equation,
$$
 Ric (\nu, \nu) = \frac{1}{2} \left( H^2 - | A|^2 - 2 K_\gamma \right) ,
$$
where $ A$ denotes the second fundamental form of $\Sigma$. Combining the above identities gives 
\begin{equation} \label{eq-bdry-static}
\Delta_\Sigma V + H \frac{\partial V}{\partial \nu} + \frac{1}{2} \left( H^2 - | A|^2 - 2 K_\gamma \right) V  = 0 . 
\end{equation} 

Because $ V $ is harmonic in $M$, by maximum principle and  $ V \to 1$ at infinity,  we may assume that $\inf_M V $ occurs on $\Sigma$ and $V$ is not a constant. Otherwise the claim $V>0$ follows easily.  Let $V(y) = \inf_M V$ for some $y\in \Sigma$. Using the Hopf boundary point lemma  and noting $V(y) = \min_\Sigma V$, we have the following inequalities  at $y$,
\begin{align*}
\frac{\partial }{\partial \nu} V (y) > 0 \quad \mbox{ and } \quad \Delta_\Sigma V (y) \ge 0.
\end{align*} 
On the other hand, the assumption on $H$ and $K_\gamma$ implies that 
$$ \frac{1}{2} \left( H^2 - | A|^2 - 2 K_\gamma \right)  \le   \frac14 H^2 - K_\gamma  \le 0,$$
Combing the above inequities and  \eqref{eq-bdry-static}, we conclude that $ V(y) > 0 $ and hence $V>0$ in $M$.

Suppose, to give a contradiction, that there is a closed, locally  area minimizing surface  in $M$. By Proposition \ref{proposition:splitting}, $g$ must be Ricci flat  in an open neighborhood of the minimal surface. Since $ V > 0 $ and $g$ is static, $g$ is analytic on $M$ (cf. \cite{Corvino:2000}). Hence, $(M, g)$ 
has vanishing Ricci curvature. In three dimensions, this implies $(M, g)$ is isometric to an exterior 
region in the Euclidean space, which is free of closed minimal surfaces. It gives a contradiction.
\end{proof}

\appendix
\section{Asymptotically flat manifolds}\label{section:asymptotically-flat}

Let $n\ge 3$. An $n$-dimensional (connected) manifold $(M,g)$ is said to be \emph{asymptotically flat}  if  $M\setminus K = \bigcup_k N_k$ for some compact subset $K\subset M$ and,  for  $q > \frac{n-2}{2}$,  there is a coordinate chart  on each \emph{end} 
\[
N_k \cong \mathbb{R}^n\setminus B_1(0)
\]
so that the components of the metric tensor satisfy 
\[
	|g_{ij} - \delta_{ij}| + |x||\partial_k g_{ij} | + |x|^2 |\partial_k \partial_\ell g_{ij}|  \le C|x|^{-q}.
\]	
We also assume the scalar curvature $R_g$ is integrable in $M$ and $g\in C^{2,\alpha}_\mathrm{loc}(M)$. 

For $r>1$, we let $B_r =  \bigcup_{x\in N_k} \{ |x|\le r\} \bigcup K$ be the closed coordinate ball with respect to the above charts, and let the coordinate sphere $S_r =  \bigcup_{x\in N_k} \{ |x|= r\}$.

Throughout this note, we follow the convention that stable minimal hypersurfaces are two-sided. We say that $M$ has a \emph{horizon boundary} if the boundary $\partial M$, possibly empty, is a disjoint union of smooth closed minimal hypersurfaces, $M$ contains no other closed minimal hypersurfaces, and we further assume that $\partial M$ is locally area minimizing if $n\ge 8$. (Note that if $3\le n \le 7$, $\partial M$ is area minimizing,  implied by  other two conditions.)

A complete non-compact hypersurface $\Sigma$ in $M$  is said to be \emph{area minimizing} if $\Sigma\cap B_r$ is a Plateau solution with the boundary spanning $\Sigma \cap S_r$ for all $r$ sufficiently large. 

We define the ADM mass of $(M, g)$ by 
\[
	m= \frac{1}{2(n-1)\omega_{n-1}} \lim_{r\to \infty} \int_{|x|=r}\sum_{i,j=1}^n (g_{ij,i}-g_{ii,j})\nu^j \, d\sigma_0,
\]
where $d\sigma_0$ is the $(n-1)$-volume measure induced from ambient Euclidean metric. We may write $m(M,g)$ to emphasize the dependence on the asymptotically flat manifold $(M, g)$.

\section{Static potential} \label{section:static}

Let $(\Omega, g)$ be an $n$-dimensional Riemannian manifold. Let $L_g^*: H^2_\mathrm{loc}(\Omega) \to L^2_{\mathrm{loc}}(\Omega)$ be a differential operator defined by 
\[
	L_g^*V = -(\Delta V)g + \nabla^2 V - V Ric,
\]	
where $\nabla^2 $ is the Hessian operator and $Ric$ is the Ricci tensor of $g$. A \emph{static potential} $V$ is a scalar valued function on $\Omega$ that satisfies $L_g^*V=0$ and is not identically zero. The equation $L_g^*V=0$ is equivalent  to the following equation
\begin{align} \label{equation:static}
	\nabla^2 V = \left(Ric - \frac{1}{n-1} R_g g\right) V.
\end{align}
By elliptic regularity, if $g\in C^{k,\alpha}_{\mathrm{loc}}$ for some $k\ge 2$, then a static potential $V$ is $C^{k,\alpha}_{\mathrm{loc}}(\Omega)$, and $V\in C^{k-2,\alpha}(\overline{\Omega})$ if $\overline{\Omega}$ is bounded (see, e.g. \cite[Proposition 2.1]{Corvino-Huang:2017} and  letting $X=0$ there). We say that $(\Omega, g)$  \emph{admits a static potential} if there is a static potential $V$ defined on $\Omega$.

\begin{lemmaA}[{\cite[Proposition 2.3 and Proposition 2.6]{Corvino:2000}}]\label{lemma:properties}
Let $(\Omega, g)$ be a connected manifold admitting a static potential $V$. Then
\begin{enumerate}
\item The scalar curvature of $g$ is constant on $\Omega$.
\item The zero set of $V$ is a  totally geodesic regular hypersurface in $\Omega$. 
\end{enumerate}
\end{lemmaA}

A static potential appears to be the only obstruction to promoting scalar curvature locally. The following statement is a special case of {\cite[Theorem 1]{Corvino:2000}}. Also see, e.g. \cite{Qing-Yuan:2016}.
\begin{theoremA}[See {\cite[Theorem 1]{Corvino:2000}}]\label{theorem:deformation}
Let $\Omega$ is a bounded open subset of a Riemannian manifold $(M, g)$ and $g\in C^{4,\alpha}(\overline{\Omega})$. Suppose $\partial \Omega$ is smooth.  If $\Omega$ does not admit a static potential, then there is a metric $\bar{g}\in C^{2,\alpha}_{\mathrm{loc}}(M)$ so that  $\bar{g} = g$ outside $\Omega$  and $R(\bar{g}) > R(g)$ in $\Omega$. 
\end{theoremA}

To analyze the asymptotics of a  static potential, we need the following ODE lemma. In this paper, we only apply the case that $Z(t)$ is real-valued, but for other future applications toward the system of Einstein constraint equations, we include the following general statement.

\begin{lemmaA}\label{lemma:ode}
Let $n\ge 1$. Let $Z: [1,\infty)\to \mathbb{R}^k$  be a $C^2$ vector valued function satisfying the differential equation
\[
	Z''(t) = A(t) Z' + B(t) Z(t), 
\]
where $A(t), B(t)$ are continuous $k\times k$ matrix functions on $[1,\infty)$ satisfying $|A(t)| + t|B(t)| \le C_1t^{-1-q}$ for some constants $C_1>0$ and $q> 0$. Then $|Z| + t|Z'|\le C_2t$ where $C_2$ depends only on $C_1$ and $Z(1), Z'(1)$. Furthermore, if  $Z$ vanishes to infinite order at infinity, i.e., for each $N>0$ there is a constant $c_N$ such that  $|Z(t)|\le c_N t^{-N}$ on $[1, \infty)$,  then $Z$ is identically zero on $[1,\infty)$.
\end{lemmaA}
\begin{proof}
Define the function $h = t^2 |Z'|^2 + |Z|^2\ge 0$. Applying  uniqueness for ODE, if $h(t)=0$ for some value of $t$, then $h$ is identically zero on $[1,\infty)$, so we may assume that $h>0$ everywhere. Compute
\begin{align*}
	h' = 2t |Z'|^2 + 2t^2 Z'\cdot  Z'' + 2Z\cdot Z'.
\end{align*}
Using the equation for $Z''$ and the bound on the coefficients, we obtain that $|h'| \le \frac{3(1+C_1)}{t} h$. Denote by $2a= 3(1+C_1)$. Solving the differential inequality yields
\[
	h(1) t^{-2a}\le h(t) \le h(1) t^{2a}.
\]
The lower bound implies that any nontrivial solution $Z$ cannot vanish to infinite order at infinity. 

The differential equation of $Z''$  implies that 
\[
	|Z''(t)| \le t^{-1}\sqrt{2h(t)} (|A(t)| + t|B(t)|).
\]
Hence, $|Z''(t)| \le C_1 \sqrt{2h(1)} t^{-2-q+a}$. By integration,  we have
\begin{align*}
	|Z'(t)| &\le |Z'(1)| + t \max_{[1,t]}|Z''| \le |Z'(1)| + C_1 \sqrt{2h(1)} t^{-1-q+a}\\
	|Z(t) | &\le |Z(1) | + t \max_{[1,t]}|Z'| \le |Z(1)| + t| Z'(1)| + C_1 \sqrt{2h(1)} t^{-q+a}
\end{align*}
and inserting these into the definition of $h(t)$ we find
\[
	h(t)\le 3(Z(1))^2 + 5t^2 |Z'(1)|^2 + 7C_1^2 h(1) t^{-2q+2a}.
\]   
It implies that the growth rate of $h$ can be further improved by a bootstrap argument, until the highest power of $t$ is quadratic. Thus,  $|Z| + t|Z'|\le C_2t$ for some constant $C_2$ depending only on $C_1, Z(1), Z'(1)$.

\end{proof}

Most of the statement in Proposition~\ref{proposition:asymptotics} below is known and can be found in \cite[Appendix C]{Beig-Chrusciel:1996} and \cite{Miao-Tam:2015}. We include the statement and the arguments here because it seems that the estimate \eqref{equation:finite}  below used in the proof of Theorem~\ref{theorem:rigidity} is not explicitly stated in the literature. 

\begin{propositionA}\label{proposition:asymptotics}
Let $(M, g)$ be an $n$-dimensional asymptotically flat manifold. Let $N$ be one of the ends. Suppose $N\cap ( B_{r_1}\setminus B_{r_0})$ admits a static potential $V$ for some $1 < r_0 < r_1$.  Then $V$ is at most of linear growth, in the sense that there is a constant $C$, depending only on $\max_{S_{r_0}} (|V| + |\nabla V|)$ and $(M, g)$, such that for each $r \in (r_0, r_1)$ 
\begin{align} \label{equation:finite}
	\sup_{B_{r}\setminus B_{r_0}} |V|\le Cr.
\end{align}
Furthermore, if $V$ is defined  on all of $N$, then either one of the following properties holds on the end $N$:
\begin{enumerate}
\item $V$ is identically zero.
\item $V = \sum_{i=1}^n a_i x^i + O^{2,\alpha}(1+|x|^{1-q}\log|x|)$ for some constants $a_1,\dots, a_n$, not all zero.
\item  $V = a_0 - a_0 m|x|^{2-n} + O^{2,\alpha}(|x|^{ 1-n} + |x|^{2-n-q}\log|x|)$, where $a_0$ is a nonzero constant and $m$ is the ADM mass of the end $N$. 
\end{enumerate}
\end{propositionA}
\begin{proof}
We compute  with respect to the polar coordinate chart of $\{x\}$ 
\begin{align*}
	V_r &= \sum_i \frac{x^i}{r}\frac{\partial V}{\partial x^i}  \\
	V_{rr} &=\sum_{i,j}  \frac{\partial^2 V}{\partial x^i \partial x^j} \frac{x^ix^j}{r^2} =\sum_{i,j} \nabla^2 V(\partial {x^i}, \partial {x^j})\frac{x^ix^j}{r^2} + \sum_{i,j,k} \Gamma_{ij}^k  \frac{x^ix^j}{r^2} \frac{\partial }{\partial x^k} V.
\end{align*}
By the static equation,  $V$ satisfies a differential equation of the form in Lemma~\ref{lemma:ode}  along each fixed angular direction, and thus  $|V| + r|V_r| + r^{2+q} |V_{rr} |\le Cr$ where $C$ depends only on the asymptotically flat metric $g$ and the values of $V, V_r$ on $S_{r_0}$. By compactness of $S_{r_0}$, the constant $C$ can be chosen uniformly among the angular directions. Therefore, we have
\[
	|V| + |x| |\partial V| + |x|^{2+q} |\partial^2 V| \le C|x|, 
\]
which, in particular, proves the first assertion. 

From now on, we assume that $V$ is defined on $N$. Since $V$ is harmonic, by the growth rate bound and harmonic expansions (e.g.  \cite[Theorem 1.17]{Bartnik:1986} and \cite{Meyers:1963}), $V$ is asymptotic to a harmonic function of homogeneous degree at most one: 
\[
	V(x) = \sum_i a_i x^i + O^{2,\alpha}(1+|x|^{1-q}\log|x|)
\]
for some constants $a_1,\dots, a_n$. If the constants are all zero, then again by the harmonic expansion,  there are constants $a_0, b$ such that $V(x) = a_0 + b|x|^{2-n} + O^{2,\alpha}(|x|^{1-n} + |x|^{2-n-q} \log|x|)$. Compute
\[
	 (\nabla^2 V)_{ij} = \frac{\partial^2 V}{\partial x^i \partial x^j} - \sum_k \Gamma_{ij}^k \frac{\partial V}{\partial x^k} = -\frac{(n-2) b\delta_{ij}}{|x|^n} + \frac{n(n-2) b x^i x^j}{|x|^{n+2}} + O^{0,\alpha}(|x|^{-1-n} + |x|^{-n-q}\log|x|).
\]
Using the alternative definition of the ADM mass (see, e.g. \cite{Miao-Tam:2016} and also \cite[Equation 1.4]{Huang:2010} for $n=3$)  and the static equation 
\[
	a_0 m =\frac{1}{(n-1)(2-n)\omega_{n-1}}\lim_{r\to \infty} \int_{\{ |x|=r\}\cap E} V R_{ij} x^i \frac{x^j}{|x|} \, d\sigma = -b.
\]
It gives the desired expansion.

If $a_0, a_1,\dots, a_n$ are all zero, then $V$ goes to zero at infinity. Applying the static equation  and bootstrapping yields that $V$ vanishes to infinite order at infinity. Applying Lemma~\ref{lemma:ode} to  the differential equation of $V$ along $r$  implies that $V$ is identically zero. 

\end{proof}

  \begin{lemmaA}\label{lemma:zero}
Let $(M, g)$ be an $n$-dimensional asymptotically flat manifold. Let $N$ be one of the ends and suppose $N$ admits a static potential $V$ with the asymptotics $V (x)= x^n + o(|x|)$ as $|x|\to \infty$.  Then there is $r_0>0$  large such that each component $\Sigma$ of $V^{-1}(0)$ in  $N\setminus B_{r_0}$ is given by a graph  $x^n = f(x^1, \dots, x^{n-1})$ and  $\Sigma$ intersects $ S_r$ transversely in a nearly equatorial $(n-2)$ sphere for $r>r_0$. 
  \end{lemmaA}
 \begin{proof}
By the previous proposition, we have $\nabla_{\partial x^n} V = 1 + O(|x|^{\gamma-1}) > 0$ for $|x|$ large,  where $ \max\{ 1-q, 0\} < \gamma < 1$. Let $x' = (x^1, \dots, x^{n-1})$.  Then by  the implicit function theorem, each  component of the zero set is given by a graph $x^n = f(x')$ with $|\nabla f|\le C|x'|^{\gamma-1} $. Then $V(x', f(x'))=0$ implies that $|f(x')| \le C|x'|^\gamma$. The constant $C$ above can be chosen uniform for all components. If $r_0$ sufficiently large, each component of $V^{-1}(0)$ intersects $S_{r}$ transversely near the equator, for all $r> r_0$. 
 \end{proof}

We include the following monotonicity formula of G. Galloway~\cite{Galloway:1993}, which is a key geometric ingredient in the proofs of Theorem~\ref{theorem:bounded} and Theorem~\ref{theorem:unbounded}.  Let $(\Omega, g)$ be an $n$-dimensional Riemannian manifold that admits a static potential. Let $\Sigma$ be a two-sided smooth hypersurface in $\Omega$. If $V>0$ in $\Omega$, let $\Phi: \Sigma \times [0, \epsilon) \to \Omega$ be the normal exponential map with respect to the conformally modified metric $\bar{g} = V^{-2}g$. In particular,  $\Phi(x, 0) = x$ and
\[
	\left.\frac{\partial }{\partial t} \Phi(x, t) \right|_{t=0} = V(x) \nu(x),
\]
where $\nu$ is the unit normal vector in the metric $g$. Let $\Sigma_t = \Phi (\Sigma \times \{ t \})$ and let $H(x,t), A(x,t)$ be the mean curvature and second fundamental form of $x\in \Sigma_t$ with respect to $\nu$ in the metric~$g$. 
\begin{lemmaA}[{Monotonicity formula~\cite[Lemma 3]{Galloway:1993}, see also \cite[Proposition 3.2]{Brendle:2013}}]  \label{lemma:monotonicity}
The mean curvature and second fundamental form of $\Sigma_t$ satisfy the following differential equality
\[
	\frac{d}{dt}\left(\frac{H}{V} \right)= |A|^2.
\]
\end{lemmaA}

\bibliographystyle{amsplain}
\bibliography{2017}

\providecommand{\bysame}{\leavevmode\hbox to3em{\hrulefill}\thinspace}
\providecommand{\MR}{\relax\ifhmode\unskip\space\fi MR }
\providecommand{\MRhref}[2]{%
  \href{http://www.ams.org/mathscinet-getitem?mr=#1}{#2}
}
\providecommand{\href}[2]{#2}
\begin{thebibliography}{10}

\bibitem{Ambrozio:2015}
Lucas Ambrozio, \emph{On static three-manifolds with positive scalar
  curvature}, arXiv:1503.03803 [math.DG], to appear in J. Differential
  Geometry.

\bibitem{Anderson-Jauregui:2016}
Michael~T. Anderson and Jeffrey~L. Jauregui, \emph{Embeddings, immersions and
  the {B}artnik quasi-local mass conjectures}, arXiv:1611.08755 [math.DG].

\bibitem{Bartnik:1986}
Robert Bartnik, \emph{The mass of an asymptotically flat manifold}, Comm. Pure
  Appl. Math. \textbf{39} (1986), no.~5, 661--693. \MR{849427 (88b:58144)}

\bibitem{Bartnik:1989}
\bysame, \emph{A new definition of quasi-local mass}, Phys. Rev. Lett.,
  vol.~62, 1989, pp.~2346 -- 2348. \MR{1056891 (91j:83032)}

\bibitem{Beig-Chrusciel:1996}
Robert Beig and Piotr~T. Chru{\'s}ciel, \emph{Killing vectors in asymptotically
  flat space-times. {I}. {A}symptotically translational {K}illing vectors and
  the rigid positive energy theorem}, J. Math. Phys. \textbf{37} (1996), no.~4,
  1939--1961. \MR{1380882 (97d:83033)}

\bibitem{Bray:2011}
Hubert~L. Bray, \emph{On the positive mass, {P}enrose, and {ZAS} inequalities
  in general dimension}, Surveys in geometric analysis and relativity, Adv.
  Lect. Math. (ALM), vol.~20, Int. Press, Somerville, MA, 2011, pp.~1--27.
  \MR{2906919}

\bibitem{Bray-Chrusciel:2004}
Hubert~L. Bray and Piotr~T. Chru\'sciel, \emph{The {P}enrose inequality}, The
  {E}instein equations and the large scale behavior of gravitational fields,
  Birkh\"auser, Basel, 2004, pp.~39--70. \MR{2098913}

\bibitem{Brendle:2013}
Simon Brendle, \emph{Constant mean curvature surfaces in warped product
  manifolds}, Publ. Math. Inst. Hautes \'Etudes Sci. \textbf{117} (2013),
  247--269. \MR{3090261}

\bibitem{Bunting-Masood-ul-Alam:1987}
Gary~L. Bunting and A.~K.~M. Masood-ul Alam, \emph{Nonexistence of multiple
  black holes in asymptotically {E}uclidean static vacuum space-time}, Gen.
  Relativity Gravitation \textbf{19} (1987), no.~2, 147--154. \MR{876598
  (88e:83031)}

\bibitem{Carlotto-Chodosh-Eichmair:2016}
Alessandro Carlotto, Otis Chodosh, and Michael Eichmair, \emph{Effective
  versions of the positive mass theorem}, Invent. Math. \textbf{206} (2016),
  no.~3, 975--1016. \MR{3573977}

\bibitem{Chrusciel:2005}
Piotr~T. Chru\'sciel, \emph{On analyticity of static vacuum metrics at
  non-degenerate horizons}, Acta Phys. Polon. B \textbf{36} (2005), no.~1,
  17--26. \MR{2125333}

\bibitem{Corvino:2000}
Justin Corvino, \emph{Scalar curvature deformation and a gluing construction
  for the {E}instein constraint equations}, Comm. Math. Phys. \textbf{214}
  (2000), no.~1, 137--189. \MR{1794269 (2002b:53050)}

\bibitem{Corvino:2017}
\bysame, \emph{A note on the bartnik mass}, To Appear Harvard CMSA Lecture
  Series, International Press, Somerville, MA, 2017 (2017).

\bibitem{Corvino-Huang:2017}
Justin Corvino and Lan-Hsuan Huang, \emph{Localized deformation for initial
  data sets with the dominant energy condition}, arXiv:1606.03078 [math.DG]
  (2016), Preprint.

\bibitem{Galloway:1993}
Gregory~J. Galloway, \emph{On the topology of black holes}, Comm. Math. Phys.
  \textbf{151} (1993), no.~1, 53--66. \MR{1201655}

\bibitem{Galloway-Miao:2017}
Gregory~J. Galloway and Pengzi Miao, \emph{Variational and rigidity properties
  of static potentials}, Comm. Anal. Geom. \textbf{25} (2017), no.~1, 163--183.

\bibitem{Huang:2010}
Lan-Hsuan Huang, \emph{Foliations by stable spheres with constant mean
  curvature for isolated systems with general asymptotics}, Comm. Math. Phys.
  \textbf{300} (2010), no.~2, 331--373. \MR{2728728 (2012a:53045)}

\bibitem{Huisken-Ilmanen:2001}
Gerhard Huisken and Tom Ilmanen, \emph{The inverse mean curvature flow and the
  {R}iemannian {P}enrose inequality}, J. Differential Geom. \textbf{59} (2001),
  no.~3, 353--437. \MR{1916951 (2003h:53091)}

\bibitem{Israel:1968}
Werner Israel, \emph{Event horizons in static electrovac space-times}, Comm.
  Math. Phys. \textbf{8} (1968), no.~3, 245--260. \MR{1552541}

\bibitem{Li:2016}
Yu~Li, \emph{Ricci flow on asymptotically euclidean manifolds},
  arXiv:1603.05336 [math.DG] (2016).

\bibitem{Meyers:1963}
Norman Meyers, \emph{An expansion about infinity for solutions of linear
  elliptic equations}, J. Math. Mech. \textbf{12} (1963), 247--264.
  \MR{0149072}

\bibitem{Miao:2005}
Pengzi Miao, \emph{A remark on boundary effects in static vacuum initial data
  sets}, Classical Quantum Gravity \textbf{22} (2005), no.~11, L53--L59.
  \MR{2145225}

\bibitem{Miao-Tam:2015}
Pengzi Miao and Luen-Fai Tam, \emph{Static potentials on asymptotically flat
  manifolds}, Ann. Henri Poincar\'e \textbf{16} (2015), no.~10, 2239--2264.
  \MR{3385979}

\bibitem{Miao-Tam:2016}
\bysame, \emph{Evaluation of the {ADM} mass and center of mass via the {R}icci
  tensor}, Proc. Amer. Math. Soc. \textbf{144} (2016), no.~2, 753--761.
  \MR{3430851}

\bibitem{Muller-zum-Hagen:1973}
H.~M\"uller~zum Hagen, David~C. Robinson, and H.~J. Seifert, \emph{Black holes
  in static vacuum space-times}, Gen. Relativity Gravitation (1973), 53--78.
  \MR{0398432}

\bibitem{Qing-Yuan:2016}
Jie Qing and Wei Yuan, \emph{On scalar curvature rigidity of vacuum static
  spaces}, Math. Ann. \textbf{365} (2016), no.~3-4, 1257--1277. \MR{3521090}

\bibitem{Robinson:1977}
D.C. Robinson, \emph{A simple proof of the generalization of {I}srael's
  theorem}, Gen. Relativity Gravitation \textbf{8} (1977), no.~8, 695--698.

\bibitem{Schoen-Yau:1979-pmt1}
Richard Schoen and Shing-Tung Yau, \emph{On the proof of the positive mass
  conjecture in general relativity}, Comm. Math. Phys. \textbf{65} (1979),
  no.~1, 45--76. \MR{526976 (80j:83024)}

\bibitem{Schoen-Yau:1981-asymptotics}
\bysame, \emph{The energy and the linear momentum of space-times in general
  relativity}, Comm. Math. Phys. \textbf{79} (1981), no.~1, 47--51. \MR{609227
  (82j:83045)}

\bibitem{Schoen-Yau:2017}
\bysame, \emph{Positive scalar curvature and minimal hypersurface
  singularities}, arXiv:1704.05490 (2017).

\end{thebibliography}
\end{document}